\documentclass[11pt]{amsart}

\usepackage{color}
\usepackage{tikz-cd}
\usepackage{tikz}
\usetikzlibrary{shapes, shapes.geometric, arrows, arrows.meta, decorations.markings, decorations.pathreplacing}
\usepackage{enumitem}
\usepackage[hypcap=false]{caption}

\usepackage{proof}
\usepackage{savesym}
\savesymbol{cir}
\savesymbol{bibsection}

\usepackage{amssymb}
\usepackage{amsmath}
\usepackage{amsfonts}
\usepackage{amsthm}
\usepackage{float}
\usepackage{times}
\usepackage{hyperref}

\floatstyle{boxed}
\newtheorem{Theorem}{Theorem}
\newtheorem*{Theorem*}{Theorem}
\newtheorem{Proposition}[Theorem]{Proposition}
\newtheorem{Lemma}[Theorem]{Lemma}
\newtheorem{Example}[Theorem]{Example}
\newtheorem{Remark}[Theorem]{Remark}
\newtheorem{Corollary}[Theorem]{Corollary}

\theoremstyle{definition}

\newtheorem{Question}[Theorem]{Question}

\newcommand{\e}{e}
\newcommand{\jn}{\vee}
\newcommand{\mt}{\wedge}
 
\newcommand{\R}{\mathbb{R}}
\newcommand{\Q}{\mathbb{Q}}
\newcommand{\Z}{\mathbb{Z}}
\newcommand{\N}{\mathbb{N}}

\newcommand{\iv}[1]{{#1}^{-1}}

\newcommand{\Si}{\mathrm{\Sigma}}

\newcommand{\cls}[1]{\mathsf{#1}}

\newcommand{\lgc}[1]{\mathrm{#1}}
\newcommand{\mdl}[1]{\models_{\lgc{#1}}}

\newcommand{\eq}{\approx}

\newcommand{\Conv}{\mathrm{Conv}}
\newcommand{\Vect}{\mathrm{Vect}}
\newcommand{\Lie}{\mathrm{Lie}}
\newcommand{\Stab}{\mathrm{Stab}}

\newcommand{\bx}{\mathbf x}
\newcommand{\by}{\mathbf y}
\newcommand{\bs}{\mathbf s}

\newcommand{\gk}{\mathfrak g}
\newcommand{\nk}{\mathfrak n}

\DeclareMathOperator{\Hom}{Hom}
\newcommand{\onto}{\twoheadrightarrow}
\newcommand{\longto}{\longrightarrow}
\newcommand{\into}{\hookrightarrow}
\newcommand{\longinto}{\lhook\joinrel\longrightarrow}
\newcommand{\abs}[1]{\left| #1 \right|}

\newcommand{\tuple}[1]{\ensuremath{\langle{#1}\rangle}}
\newcommand{\sgr}[1]{{#1}^+} 
\newcommand{\nsgr}[1]{{#1}^\circ}

\title{Ordering groups and the Identity Problem}

\author{Corentin Bodart, Laura Ciobanu, George Metcalfe}
\address{Section de Math{\'e}matiques, Universit{\'e} de Gen{\`e}ve, Switzerland}

\address{Department of Mathematics, Heriot-Watt University, and the Maxwell Institute for Mathematical Sciences, Edinburgh, UK}

\address{Mathematical Institute, University of Bern, Switzerland}
\begin{document}

\begin{abstract} 
In this paper, the Identity Problem for certain groups, which asks if the subsemigroup generated by a given finite set of elements contains the identity element, is related to problems regarding ordered groups. Notably, the Identity Problem for a torsion-free nilpotent group corresponds to the problem asking if a given finite set of elements extends to the positive cone of a left-order on the group, and thereby also to the Word Problem for a related lattice-ordered group. 

A new (independent) proof is given showing that the Identity and Subgroup Problems are decidable for every finitely presented nilpotent group, establishing  also the decidability of the Word Problem for a family of lattice-ordered groups. A related problem, the Fixed-Target Submonoid Membership Problem, is shown to be undecidable in nilpotent groups. 

Decidability of the Normal Identity Problem (with  `subsemigroup'  replaced by `normal subsemigroup') for free nilpotent groups is established using the (known) decidability of the Word Problem for certain lattice-ordered groups. Connections between orderability and the Identity Problem for a class of torsion-free metabelian groups are also explored.
\end{abstract}

\maketitle


\section{Introduction} \label{s:introduction}

This paper showcases the interplay between topics of different flavours in group theory --- notably, orderability of groups, membership problems in groups, and word problems for lattice-ordered groups --- and uses a range of tools, from convex geometry to Lie algebras, to establish these connections.

The {\em Identity Problem} for a group $G$ asks if the subsemigroup generated by a given finite set of elements of $G$ contains the identity element $e_G$. This problem has been studied in~\cite{BP2010,BHP2017,Dong2022,DongLICS2023} for various matrix groups, and for certain classes of nilpotent and metabelian groups in~\cite{Dong2023, DongSODA2024, DongSTOC2024}; see~\cite{dong2023recent, lohrey2024membership} for comprehensive surveys.  \emph{The Subgroup Problem}, which asks if the subsemigroup generated by a given finite set of elements of $G$ is a subgroup, has been studied alongside the Identity Problem for a similar range of groups.

Our first goal in this paper is to explain how the Identity Problem for $G$ relates (in some settings) to both the {\em Left-Order Extension Problem} for $G$, which asks if a given finite set of elements of $G$ extends to the positive cone of a left-order on $G$, and the {\em Word Problem} for a corresponding lattice-ordered group. Left-orders on groups --- in particular, their dynamical realizations and associated toplogical spaces --- are an active research area (see, e.g.,~\cite{CR16,DeroinNavasRivas}), while word problems for lattice-ordered groups have been studied widely in the ordered groups literature (see, e.g.,~\cite{HM79,McC82,Kop82,BG09}). We also explain how the  {\em Normal Identity Problem} for $G$, with `subsemigroup' replaced by `normal subsemigroup', relates (in some cases) to both the {\em Bi-Order Extension Problem} for $G$, with `left-order' replaced by `bi-order', and the decidability of the equational theory of a corresponding variety of lattice-ordered groups. 

Our second goal is to address these problems in the setting of {\em nilpotent groups}. In particular, we provide a new (independent) proof that the Identity and Subgroup Problems are decidable for every finitely generated nilpotent group, concluding:

\begin{Theorem}\label{t:main} \:
\begin{enumerate}[leftmargin=8mm, label={\normalfont(\alph*)}]
\item The Identity Problem is decidable for every finitely generated nilpotent group $G$. That is, given a finite set $S$ of elements in $G$, it is possible to decide whether the subsemigroup generated by $S$ contains the identity element $e_G$.
\item The Subgroup Problem is decidable for any finitely generated nilpotent group $G$. That is, given a finite set $S$ of elements in $G$, it is possible to decide whether the subsemigroup generated by $S$ is a subgroup of $G$.
\end{enumerate}
\end{Theorem}

\noindent
Although decidability of the Subgroup Problem implies decidability of the Identity Problem, we state them separately in the theorem, as an alternative way of solving the Identity Problem (but not the Subgroup Problem) is possible (see Section~\ref{s:nilpotent_groups}). Theorem~\ref{t:main} relies on Proposition~\ref{p:discrete_Chow}, our main technical tool, which establishes the equivalence of the required algebraic statements about a nilpotent group and convex geometry properties of the group. In particular, Proposition~\ref{p:discrete_Chow} allows us to detect finite-index subgroups among subsemigroups. An equivalent result was independently proven by Shafrir~\cite[Theorem 2]{shafrir2024saturation}. During the writing of this paper Theorem~\ref{t:main} has been generalized by the first author and Dong to virtually solvable matrix groups over algebraic numbers \cite{BodartDong2024}.

As consequences of Theorem~\ref{t:main}, we obtain:

\newtheorem*{c:nilpotent}{Corollary~\ref{c:nilpotent}}
\begin{c:nilpotent}
Let $\tuple{X\mid R}$ be a finite presentation of a torsion-free nilpotent group. Then the $\ell$-group with finite presentation $\tuple{X\mid R}$ has a decidable Word Problem.
\end{c:nilpotent}

\newtheorem*{c:normal}{Corollary~\ref{c:normal}}
\begin{c:normal}
The Normal Identity Problem and Bi-Order Extension Problem are decidable for any finitely generated free nilpotent group of class~$c$.
\end{c:normal}

Considering the strong links between orderability and membership problems in nilpotent groups exhibited above, a natural question is whether similar results hold for other groups. We show that the class of groups $G_\lambda\le\mathrm{Aff}_+(\Q)$ consisting of orientation-preserving affine transformations of the rationals (which are torsion-free metabelian) are fully left-orderable (Corollary~\ref{cor:affine_fully}), and so the question of extending a finite set to a left-order is equivalent to the Identity Problem. While the Identity Problem is known to be decidable in metabelian groups, this equivalence gives a much faster algorithm for the Identity Problem. Section~\ref{s:aff} is a starting point to exploring metabelian groups more generally, and finding further groups that are fully (left-)orderable.

\subsection*{Related membership problems.} Theorem~\ref{t:main} stands in contrast to the undecidability of the \emph{Submonoid Membership Problem} in finitely generated nilpotent groups: for certain nilpotent groups $G$, it is undecidable for an arbitrary finite set $S \subset G$ and element $g \in G$, whether the submonoid generated by $S$ contains $g$ (\cite{romankov2022undecidability}). 

Theorem~\ref{t:main} also stands in contrast to the undecidability of the \emph{Fixed-Target Submonoid Membership Problem} in finitely generated nilpotent groups, proved in Proposition~\ref{fixed_target}. This asks whether given a nilpotent group $G$ together with an element $g_0\in G$, it is possible to decide if $g_0$ is in the submonoid generated by $S$, where $S \subset G$ is an arbitrary finite set. This is an intermediate problem between the Identity Problem on one hand, where $g_0=e_G$ and we ask about membership in a subsemigroup, and the Submonoid Membership Problem on the other hand, where both $S$ and $g_0$ are arbitrary. 

\subsection*{Complexity of our algorithms.} We establish polynomial time complexity for the Identity and Subgroup Problems for a fixed nilpotent group $G$ in Proposition~\ref{p:complexity}. Although this has already been claimed in~\cite{shafrir2024saturation} and~\cite{lohrey2024membership}, no details were given in those references. We also discuss in Section~\ref{sec:complexity} why the \emph{uniform} versions of these problems, when a group presentation for some group $G$ is part of the input, are not polynomial with the algorithms currently available.

\subsection*{Structure of the paper} We give the background on orderability and connections of orderability to the Identity Problem in nilpotent groups in Section~\ref{s:orders_and_identity}, connecting these with word problems for $\ell$-groups in Section~\ref{s:lgroups}. Section~\ref{s:convex_geom} contains the main technical results, linking convex geometry to subsemigroups in nilpotent groups, results that are used in Section~\ref{s:nilpotent_groups} to establish Theorem~\ref{t:main}. In Section~\ref{sec:complexity} we discuss the computational complexity of the Identity and Subgroup Problems. 

Section~\ref{sec:fixed-target} is concerned with the Fixed-Target Submonoid Membership Problem, which is proved to be undecidable in nilpotent groups. Finally, Section~\ref{s:aff} considers an additional class of groups for which full orderability holds, groups of orientation-preserving affine transformations of the rationals, and connects this to the Identity Problem, shifting the focus from nilpotent groups in earlier sections to metabelian ones.
We finish the paper with a few open questions arising from our results on orderability for metabelian groups.

%


\section{Extending partial orders on groups and the Identity Problem} \label{s:orders_and_identity}

Let $G$ be any group with identity element $\e_G$. A partial order $\le$ on $G$ is called a {\em partial left-order} on $G$ if it is left-invariant, i.e., $a \le b$ implies $ca \le cb$ for all $a,b,c \in G$; if $\le$ is also total, it is called a {\em left-order}. A (partial) left-order on $G$ is called a {\em (partial) bi-order} if it is both left- and right-invariant, i.e., $a \le b$ implies $cad \le cbd$ for all $a,b,c,d\in G$.

The {\em positive cone} $P := \{a \in G \mid \e_G < a\}$ of any partial left-order $\le$ on $G$ is a subsemigroup of $G$ that omits $\e_G$, i.e., $a,b\in P$ implies $ab\in P$, and $\e_G\not\in P$. Conversely, given any subsemigroup $P$ of $G$ that omits $\e_G$, the binary relation $\le$ on $G$ defined by $a \le b\,:\Longleftrightarrow\, \iv{a}b \in P \cup \{\e_G\}$ is a partial left-order on $G$ with positive cone $P$. Because of this equivalence we will use the term `order' to refer to either an order $\le$ or its positive cone. 

Partial left-orders on $G$ can be identified with subsemigroups of $G$ that omit $\e_G$, and a subset $S$ of $G$ extends to a partial left-order on $G$ if and only if $\sgr{S}$, the subsemigroup generated by $S$, omits $\e_G$.  

A partial left-order $P$ on $G$ is a partial bi-order if and only if it is {\em normal}, i.e., $a\in P$ and $b\in G$ implies $ba\iv{b}\in P$. Partial bi-orders on $G$ can therefore be identified with normal subsemigroups of $G$ that omit $\e_G$, and a subset $S$ of $G$ extends to a partial bi-order on $G$ if and only if $\nsgr{S}$, the normal subsemigroup generated by $S$, omits $\e_G$. 

If there exists a left-order (bi-order) on $G$, then $G$ is said to be {\em left-orderable} ({\em bi-orderable}). We call $G$ {\em fully left-orderable} ({\em fully bi-orderable}) if every partial left-order (bi-order) on $G$ extends to a left-order (bi-order) on $G$. Let us refer to the problem of deciding if a given finite subset of $G$ extends to a left-order (bi-order) on $G$ as the {\em Left-Order (Bi-Order) Extension Problem}. The following results are direct consequences of the preceding definitions and remarks.
\begin{Proposition}\label{p:fully-orderable-identity} Let $G$ be a group.
\begin{enumerate}[leftmargin=8mm, label={\normalfont(\alph*)}]
\item If $G$ is fully left-orderable, then the Identity and  Left-Order Extension Problems for $G$ are complements: a finite subset $S$ of $G$ extends to a left-order on $G$ if and only if $\e_G\not\in\sgr{S}$.
\item If $G$ is fully bi-orderable, then the Normal Identity and Bi-Order Extension Problems for $G$ are complements: a finite subset $S$ of $G$ extends to a bi-order on $G$ if and only if $\e_G\not\in\nsgr{S}$.
\end{enumerate}
\end{Proposition}
In order to apply these results in the setting of nilpotent groups, we make use of the following theorems from the literature on ordered groups:

\begin{Theorem}[Malcev~\cite{Mal51}]
Every torsion-free nilpotent group is fully bi-orderable.
\end{Theorem}

\begin{Theorem}[{Rhemtulla~\cite[Theorem~4]{Rhe72}}]\label{t:Rhemtulla}
Every torsion-free nilpotent group is fully left-orderable.
\end{Theorem}

Combining these two theorems with Proposition~\ref{p:fully-orderable-identity} yields:

\begin{Corollary}
Let $G$ be any torsion-free nilpotent group.
\begin{enumerate}[leftmargin=8mm, label={\normalfont(\alph*)}]
\item The Identity and Left-Order Extension Problems  are complements for $G$: a finite subset $S$ of $G$ extends to a left-order on $G$ if and only if $\e_G\not\in\sgr{S}$.
\item The Normal Identity and Bi-Order Extension Problems are complements for $G$: a finite subset $S$ of $G$ extends to a bi-order on $G$ if and only if $\e_G\not\in\nsgr{S}$.
\end{enumerate}
\end{Corollary}

Theorem~\ref{t:main} establishes that the Identity Problem, and hence also the Left-Order Extension Problem, are decidable for any finitely generated nilpotent group, noting that the latter is trivial for any group with torsion. To prove this theorem, we will use the fact that every left-order on a torsion-free nilpotent group possesses a key property. A left-order $P$ on a group $G$ is said to be \emph{Conradian} if $a,b\in P$ implies $\iv{b}ab^n\in P$ for some $n\in\N^+$. 

\begin{Theorem}[{Rhemtulla~\cite[Theorem 2]{Rhe72}}]
\label{t:secondRhemtulla}
Every left-order on a torsion-free nilpotent group is Conradian.
\end{Theorem}

The following fundamental property of Conradian orders is known as Conrad's theorem; for a modern treatment of the property, see~\cite[Corollary 3.2.28]{DeroinNavasRivas}.

\begin{Theorem}[{Conrad~\cite[Theorem 4.1]{conrad}}]\label{t:conrad}
Let $G$ be any finitely generated group and let $P$ be a Conradian order on $G$. Then there exists a homomorphism $f\colon G\to\R$ such that $a\in P$ implies $f(a)\ge 0$.
\end{Theorem}

\begin{Corollary} \label{crl:easy_out}
Let $G$ be any finitely generated torsion-free nilpotent group and let $S\subseteq G$. If $\e_G\notin\sgr{S}$, then there exists a non-zero homomorphism $f\colon G\to\R$ such that $f(a)\ge 0$ for all $a\in S$. 
\end{Corollary}
\begin{proof}
Suppose that $\e_G\notin\sgr{S}$. Then $\sgr{S}$ is a partial left-order on $G$ that extends to a left-order $P$ on $G$, by Theorem~\ref{t:Rhemtulla}. But $P$ is Conradian, by Theorem~\ref{t:secondRhemtulla}, so the existence of a suitable homomorphism $f$ follows from Theorem~\ref{t:conrad}.
\end{proof}


\section{Word problems for $\ell$-groups and the Identity Problem} \label{s:lgroups}


A {\em lattice-ordered group} ({\em $\ell$-group}) is an algebraic structure $\tuple{L,\mt,\jn,\cdot,\iv{},\e}$  such that $\tuple{L,\cdot,\iv{},\e}$ is a group, $\tuple{L, \mt, \jn}$ is a lattice with a lattice-order defined by $a \le b\,:\Longleftrightarrow\,a\mt b=a$, and $a \le b$ implies $cad \le cbd$ for all $a,b,c,d \in L$. The class $\cls{LG}$ of $\ell$-groups forms a variety (see, e.g.,~\cite{AF88} for details). Key examples of $\ell$-groups are provided by sets  $Aut(\tuple{C,\preceq})$ of order-preserving permutations of a totally ordered set $\tuple{C,\preceq}$, with the group operation given by composition of functions and the lattice-order defined by setting $f\leq g\,:\Longleftrightarrow\,f(a) \preceq g(a)$ for all $a \in C$. In particular,  $Aut(\tuple{\Z,\le})$ is isomorphic to the abelian $\ell$-group $\tuple{\Z, \min, \max, +,-, 0}$. A fundamental result of Holland states that every $\ell$-group embeds into $Aut(\tuple{C,\preceq})$, for some totally ordered set $\tuple{C,\preceq}$~\cite{Hol63}.

For any class $\cls{K}$ of $\ell$-groups and set of $\ell$-group equations $\Si\cup\{s\eq t\}$, let $\Si\mdl{\cls{K}}s\eq t$ denote that for any homomorphism $h$ from the algebra of $\ell$-group terms to some $L\in\cls{K}$, if $h(s')=h(t')$ for all $s'\eq t'\in\Si$, then $h(s)=h(t)$. 

The Word Problem for an $\ell$-group with presentation $\tuple{X\mid R}$ is decidable if and only if there exists an algorithm to decide if $\{r\approx e\mid r\in R\}\models_\cls{LG} s\eq t$ for any $\ell$-group equation $s\eq t$ (see, e.g.,~\cite{MPT23}). Let us note also that deciding such consequences in $\cls{LG}$ can be effectively reduced to deciding in $\cls{LG}$ finitely many consequences of the form $\{r\approx e: r\in R\}\models_\cls{LG}e\leq t_1\vee \dots \vee t_n$, where  $t_1,\dots,t_n$ are group terms.

For a left-orderable group with a fixed presentation, the existence of a left-order containing a given finite set of elements corresponds to a consequence in $\cls{LG}$.

\begin{Theorem}[{\cite[Theorem~3]{CM19}}]\label{t:CM19a}
Let $G$ be a left-orderable group with presentation $\tuple{X\mid R}$. For any elements $[t_1],\dots,[t_n]$ in~$G$ corresponding to elements $t_1,\dots,t_n$ of the free group over $X$, the following are equivalent:
\begin{enumerate}[leftmargin=8mm, label={\normalfont(\arabic*)}]
\item $\{[t_1],\dots,[t_n]\}$ does not extend to a left-order on $G$;
\item $\{r\approx e: r\in R\}\models_\cls{LG}\e\leq t_1\vee \dots \vee t_n$.
\end{enumerate}
\end{Theorem}

\noindent
Combining this result with Proposition~\ref{p:fully-orderable-identity} yields the following relationship between the Identity Problem for a fully left-orderable group with a fixed presentation and the Word Problem for a corresponding $\ell$-group:

\begin{Corollary}
Let $G$ be a fully left-orderable group with presentation $\tuple{X\mid R}$. 
\begin{enumerate}[leftmargin=8mm, label={\normalfont(\alph*)}]
\item For any elements $[t_1],\dots,[t_n]$ in~$G$ corresponding to elements $t_1,\dots,t_n$ of the free group over $X$,
\begin{align*}
\qquad\quad\e_G\in\sgr{\{[t_1],\dots,[t_n]\}}\iff\{r\approx e: r\in R\}\models_\cls{LG}\e\leq t_1\vee \dots \vee t_n.
\end{align*}
\item The Identity Problem for $G$ is decidable if and only if the Word Problem for the $\ell$-group with presentation $\tuple{X\mid R}$ is decidable.
\end{enumerate}
\end{Corollary}

Since finitely generated torsion-free nilpotent groups are fully left-orderable (Theorem~\ref{t:Rhemtulla}) and have a decidable Identity Problem  (Theorem~\ref{t:main}), we obtain:

\begin{Corollary}\label{c:nilpotent}
Let $\tuple{X\mid R}$ be a finite presentation of a torsion-free nilpotent group. Then the $\ell$-group with finite presentation $\tuple{X\mid R}$ has a decidable Word Problem.
\end{Corollary}

Let us turn our attention next to {\em o-groups} --- $\ell$-groups with an underlying total order --- and subdirect products of o-groups, referred to as {\em representable} $\ell$-groups. The class of representable $\ell$-groups forms a variety of $\ell$-groups that is axiomatized relative to $\cls{LG}$ by the equation $\e\le x\jn y\iv{x}\iv{y}$ (see, e.g.,~\cite{AF88} for details).

Now let $\Si$ be any set of group equations. We denote by $F_\Si(X)$ the relatively free group over a set $X$ of the variety of groups satisfying $\Si$, and let $\cls{Rep}_\Si$ denote  the variety of representable $\ell$-groups satisfying $\Si$. In the case where $F_\Si(X)$ is bi-orderable, the existence of a bi-order  containing a given finite set of elements corresponds to the validity of an equation in $\cls{Rep}_\Si$.

\begin{Theorem}[{\cite[Theorem~8]{CM19}}]\label{t:bi-orderable}
Let $\Si$ be a set of group equations and let $X$ be a set such that $F_\Si(X)$ is bi-orderable. For any $[t_1],\dots,[t_n]\in F_\Si(X)$ corresponding to  $t_1,\dots,t_n\in F(X)$, the following are equivalent:
\begin{itemize}
\item[\rm (1)] $\{[t_1],\ldots,[t_n]\}$ does not extend to a bi-order on $F_\Si(X)$;
\item[\rm (2)] $\cls{Rep}_\Si\models \e \le t_1\lor\cdots\lor t_n$.
\end{itemize}
\end{Theorem}

\noindent
Combining this result with Proposition~\ref{p:fully-orderable-identity} yields the following:

\begin{Corollary}\label{c:normalidentity}
Let $\Si$ be a set of group equations.
\begin{enumerate}[leftmargin=8mm, label={\normalfont(\alph*)}]
\item For any set $X$ such that $F_\Si(X)$ is fully bi-orderable and elements $[t_1],\dots,[t_n]$ in $F_\Si(X)$ corresponding to  $t_1,\dots,t_n\in F(X)$,
\begin{align*}
\e\in\nsgr{\{[t_1],\dots,[t_n]\}}\iff\cls{Rep}_\Si\models \e \le t_1\lor\cdots\lor t_n.
\end{align*}
\item If $F_\Si(\omega)$ is fully bi-orderable, then the Normal Identity Problem for $F_\Si(\omega)$ is decidable if and only if $\cls{Rep}_\Si$ has a decidable equational theory.
\end{enumerate}
\end{Corollary}

We apply these results again in the setting of nilpotent groups, where $\Si$ is taken to be a set of group equations that defines the variety of nilpotent groups of class~$c$. An $\ell$-group is said to be {\em nilpotent of class~$c$} if its group reduct is nilpotent of class~$c$, and we let $\cls{NLG}_c$ denote the variety of all such $\ell$-groups. We make use of the following result:

\begin{Theorem}[Koptytov~\cite{Kop82}]\label{t:kopytov}
Any nilpotent $\ell$-group of class~$c$ that is finitely presented in $\cls{NLG}_c$ has a decidable Word Problem.
\end{Theorem}

In particular, every finitely generated free nilpotent $\ell$-group of class~$c$ is finitely presented in $\cls{NLG}_c$, so the equational theory of $\cls{NLG}_c$ is decidable. Since nilpotent $\ell$-groups are representable (see~{\cite{Kop82}), and free nilpotent groups of class~$c$ are fully bi-orderable, Corollary~\ref{c:normalidentity} and Theorem~\ref{t:kopytov} yield:

\begin{Corollary}\label{c:normal}
The Normal Identity and Bi-Order Extension Problems are decidable for any finitely generated free nilpotent group of class~$c$.
\end{Corollary}


\section{Convex geometry and subsemigroups of nilpotent groups} \label{s:convex_geom}


The proof of Theorem~\ref{t:main} will rely on Proposition~\ref{p:discrete_Chow}, which combines algebraic characterizations of subsemigroups in nilpotent groups with convex geometry. In order to state and prove Proposition~\ref{p:discrete_Chow}, we first give some background on convex geometry and polytopes.

\subsection{Preliminaries}
A \emph{linear form} (or functional) $f\colon \R^d\to \R$ has the form $f(\bx)=\textbf{a} \cdot \bx$, where $\textbf{a}\in \R^d$. A \emph{hyperplane} is the set $\{\bx \in \R^d \mid f(\bx)=b\}$ given by some non-zero linear form $f$ and $b\in \R$.
		
A subset $C\subseteq \R^d$ is \emph{convex} if $\lambda \bx+(1-\lambda)\by\in C$ for all $\bx,\by\in C$ and $\lambda\in [0,1]$. Given a subset $S\subseteq \R^d$, its \emph{convex hull} is the smallest convex set containing $S$. Equivalently, the convex hull is the set of \emph{convex combinations} of elements of $S$:
\[
\Conv(S) := \left\{ \lambda_1\bs_1+ \cdots + \lambda_k\bs_k \;\middle|\; k\in\N, \lambda_i\in\R_{\ge 0} \text{ s.t.\ }\sum_{i=1}^k\lambda_i=1, \bs_i\in S \right\}.
\]	
Interesting examples are polytopes, they can be defined in two equivalent ways: 
	\begin{enumerate}[leftmargin=8mm, label=\arabic*.]
		\item A \emph{convex polytope} is the convex hull of finitely many points in $\R^d$.
		\item A \emph{convex polytope} is a bounded subset of $\R^d$ that can be written as the intersection of finitely many half-spaces $\{\bx\in \R^d \mid f_i(\bx)\le a_i\}$, where $f_i\colon\R^d\to\R$ are non-trivial linear forms and $a_i\in\R$.
	\end{enumerate}
Thus any polytope $P$ can be described in two ways:
\begin{enumerate}[leftmargin=8mm, label=\arabic*.]
	\item A \emph{V-representation} for $P$ is a finite set of points whose convex hull is $P$. 

	\item A \emph{H-representation} for $P$ is a finite set of half-spaces $\{f_i(\bx)\le a_i\}$ whose intersection is $P$. 
\end{enumerate}

It is algorithmically feasible to produce H-representations from V-representations (and vice versa). This is a classical problem, the \emph{Facet Enumeration Problem}, solved in general using Fourier-Motzkin elimination. See for instance~\cite[\S 1.2]{Ziegler}.

A \emph{face} of $P$ is the set of the form $\{\bx\in P\mid f(\bx)=b\}$, where $P$ is included in the half-space $\{\bx\in\R^d:f(\bx)\le b\}$. A \emph{facet} is a face of dimension $d-1$.

\begin{Theorem}[Carath\'eodory, see~{\cite[Prop.\ 1.15]{Ziegler}}] \label{t:CT} 
Fix $S\subseteq \R^d$. For each point $\bx\in \Conv(S)$, there exist $d+1$ elements $\bs_0,\bs_1,\ldots,\bs_d\in S$ such that
\[ \bx\in \Conv(\{\bs_0,\bs_1,\ldots,\bs_d\}). \] 
\end{Theorem}

\begin{Theorem}[Hahn-Banach Separation Theorem] \label{t:HB} Let $C\subset \R^d$ be a convex set and $\mathbf p\in \R^d\setminus C$. There exists a non-zero linear form $f\colon \R^d\to \R$ such that $f(\bx)\ge f(\mathbf p)$ for all $\bx\in C$.
\end{Theorem}

Finally, a result of relevance for later proofs is the following. 

\begin{Lemma} \label{lem:fi_iff_cocompact}
	A subgroup $H\leq \Z^r$, $r\geq 1$, has infinite-index if and only if there exists a non-zero homomorphism $f\colon \Z^r\to \Z$ such that $H\leq\ker(f)$. In particular, this is equivalent to $H$ belonging to a hyperplane.
\end{Lemma}
\begin{proof}
	For the right-to-left direction, if $H\le\ker(f)$ for some non-zero homomorphism $f\colon\Z^r\to \Z$, then $[\Z^r:H]\geq [\Z^r:\ker(f)]=|\mathrm{Im}(f)|=\infty$.
	
	\smallskip
	
	For the left-to-right direction, if $H$ has infinite-index, then $\Z^r/H$ is an infinite, finitely generated, abelian group, and hence it factors onto $\Z$: the composition $\Z^r\onto\Z^r/H\onto\Z$ provides the desired homomorphism $f$. Thus $f(H)=0$, and so $H$ belongs to the hyperplane determined by $f$ and $b=0$.
\end{proof}

\subsection{Main technical result (Proposition~\ref{p:discrete_Chow})} Given a group $G$, we define the characteristic subgroup
	\[ I_G([G,G]) := \left\{g\in G \;\big|\; \exists n\in\N^+, g^n\in [G,G]\right\}, \]
	 called the \emph{isolator} of $[G,G]$. The quotient $G/I_G([G,G])$ is the largest torsion-free abelian quotient of $G$. If $G$ is finitely generated, then $G/I_G([G,G])\simeq\Z^r$, where $r\geq 0$ is the torsion-free rank of the abelianisation of $G$. To see this, consider the projection of $G$ onto its abelianisation, that is $G\onto G/[G,G]\simeq\Z^r\times T$ with $T$ finite, and notice that $I_G([G,G])$ is the preimage of $T$ under this projection.
	 


\begin{Proposition} \label{p:discrete_Chow}
	Let $G$ be a finitely generated nilpotent group, and consider the map $\pi\colon G\onto G/I_G([G,G])\simeq\Z^r$ with $r \geq 1$. For any set $S\subseteq G$, the following are equivalent:
	\begin{enumerate}[leftmargin=8mm, label={\normalfont(\arabic*)}]
		\item $\Conv(\pi(S))\subseteq \R^r$ contains a ball $B(\mathbf 0,\varepsilon)$ for some $\varepsilon>0$.
		\item For every non-zero linear form $f\colon \Z^r\to\R$, there exists $s\in S$ such that \ $f(\pi(s))<0$. If $S$ is finite, we can restrict to $f\colon \Z^r\to \Z$.
		
		\item For every non-zero homomorphism $\phi\colon G\to \R$, there exists $s\in S$ such that \ $\phi(s)<0$. 
		
		\item The subsemigroup $S^+$ is a finite-index subgroup of $G$.
	\end{enumerate}
\end{Proposition}
\begin{proof}
	(1) $\Rightarrow$ (2) is trivial, and $\neg$(1) $\Rightarrow$ $\neg$(2) is the Hahn-Banach Theorem (the version stated in Theorem~\ref{t:HB}), taking $C$ to be the interior of $\Conv(\pi(S))$ and $p=\mathbf 0$. Note that if $S$ is finite we can take as $f\colon\Z^r\to\Z$ a linear form defining one of the facets of $\Conv(\pi(S))\subset\Z^r$. The equivalence (2) $\Leftrightarrow$ (3) follows from the fact that the map
	\[ \pi^*\colon \begin{pmatrix}
		\Hom(G/ I_G([G,G]),\R) & \longto & \Hom(G,\R) \\
		f & \longmapsto & f\circ\pi 
	\end{pmatrix}\]
	is a bijection. The map $\pi^*$ is clearly injective; it is also onto, as any homomorphism $\phi\colon G\to \R$ factors through $G/I_G([G,G])\simeq\Z^r$ since $\R$ is torsion-free and abelian. (4) $\Rightarrow$ (2) follows from Lemma~\ref{lem:fi_iff_cocompact} with $H=\pi(S^+)$: we have 
	\[ |\Z^r:H|=|\pi(G):\pi(S^+)|\le |G:S^+| <\infty, \]
	and because $f\colon \Z^r\to\Z$ is non-zero, by Lemma~\ref{lem:fi_iff_cocompact} there must exist $h \in H$ such that $f(h)< 0$. Since $h \in \pi(S)^+$, there exists $s\in S$ such that $f(\pi(s)) < 0$.
	
	It remains to prove (1, 2, 3) $\Rightarrow$ (4). We argue by induction on the nilpotency class $c$ of $G$. Let $P := \Conv(\pi(S))$.
	\medbreak
	\noindent \textbf{Base case.} We start with $c=1$. That is, we may assume that $G\simeq \Z^r\times T$ for some finite abelian group $T$, and let $\pi\colon \Z^r\times T \to \Z^r$. 
	We prove that $-\pi(s)\in \pi(S)^+$ for all $s\in S$, and hence that $\pi(S)^+$ is a subgroup (as $S$ is non-empty).
	
\begin{itemize}[leftmargin=6mm]
	\item We first assume that $S$ is finite, so $P$ is a convex polytope. Consider the ray (or half-line) from $\mathbf 0$ through $-\pi(s)$. This ray intersects some facet $F$ of $P$ at $-x\pi(s)$ for some $x>0$ as $B(\mathbf 0,\varepsilon) \subset P$ for some $\varepsilon>0$ (by (1)). Using Carath{\'e}odory's Theorem (Theorem~\ref{t:CT}), there exist $r$ vertices $\pi(s_1),\ldots,\pi(s_r)$ of $F$ and $y_1,\ldots,y_r\in\R_{\ge 0}$ summing to $1$ such that 
	\[ -x\pi(s) = y_1\pi(s_1) +  \ldots + y_r\pi(s_r). \]
	Moreover, we may assume $x\in\mathbb Q_{>0}$ and $y_i\in\Q_{\ge 0}$, as the coefficients of the underlying system are integers.  We then multiply by some well-chosen positive integer $N$ in order to cancel out denominators, and get
	\[ -\pi(s) = (Nx-1)\cdot \pi(s) + Ny_1\cdot \pi(s_1) + \ldots + Ny_r\cdot \pi(s_r) \in \pi(S)^+. \]
	\item If $S$ is infinite, we can find some finite subset $S_0\subset S$ such that $\mathbf 0$ lies in the interior of $P_0 := \Conv(\pi(S_0))$ and $S_0\ni s$. Now we may repeat the previous argument with $S_0$ and conclude that $-\pi(s)\in \pi(S_0)^+\subseteq \pi(S)^+$.
\end{itemize}
	This proves that $\pi(S)^+$ is a subgroup. The same holds for $S^+$: for each $s\in S$, $s=\pi(s)t$, for some $t\in T$ (we abuse notation and write $s=\pi(s)t$ instead of $s=(\pi(s), e_T)(e_{\Z^r},t)$), and so $s \pi(s^{-1}) \in T$. We can write $\pi(s^{-1}) = v \in \pi(S)^+$ and since $v=w t'$ for some $w\in S^+, t' \in T$, there exists $w\in S^+$ such that $sw\in T$ hence $(sw)^{|T|}=e_G$. This implies $s^{-1}=s^{\abs T -1}w^{\abs T} \in S^+$.
	
	Finally, $\pi(S)^+$ is not included in any hyperplane (condition (2)); hence $\pi(S)^+$ has finite-index in $\Z^r$ by Lemma~\ref{lem:fi_iff_cocompact}, and therefore $S^+$ has finite-index in $G$.
	\bigbreak
	
	\noindent \textbf{Induction step.} Suppose that the induction hypothesis holds for $c-1\ge 1$. We fix $G$ of nilpotency class $c$, that is, the $(c+1)$-th term in the lower central series $\gamma_{c+1}(G)$ is trivial. We also fix a subset $S\subseteq G$ satisfying the equivalent conditions (1, 2, 3). We prove that (i) $S^+\cap \gamma_c(G)$ is a finite-index subgroup in $\gamma_c(G)$, and deduce that (ii) $S^+$ is a finite-index subgroup in $G$.
	
	\medskip
	
	(i) Proof that $S^+\cap\gamma_c(G)$ is a finite-index subgroup in $\gamma_c(G)$.
	
	\smallskip
	
	\noindent Since $\gamma_c(G)$ is finitely generated abelian and $S^+\cap \gamma_c(G)=(S^+\cap \gamma_c(G))^+$, we could use the base case. We verify condition (3): for each non-zero homomorphism $\phi\colon \gamma_c(G)\to \R$, we construct $t\in S^+\cap \gamma_c(G)$ such that $\phi(t)<0$ (Claim 4).
	
	\medskip

	\emph{Claim 1.} $H := S^+\,\gamma_c(G)$ is a finite-index subgroup of $G$.
	
	\smallskip
	
	\noindent For any non-zero homomorphism $g'\colon G/\gamma_c(G)\to\R$, the composition \vspace*{-2mm}
	\begin{center}
		\begin{tikzcd}
			G \arrow[r, two heads, "\tau"] & G/\gamma_c(G) \arrow[r,"g'"] & \R
		\end{tikzcd}
	\end{center}\vspace*{-2mm}
	is a non-zero homomorphism and hence there exists $s\in S$ such that \ $g'(\tau(s))<0$. This means that $\tau(S)\subseteq G/\gamma_c(G)$ satisfies condition (3). However, $G/\gamma_c(G)$ has nilpotency class $c-1$, and so, by the induction hypothesis, $\tau(S)^+$ is a finite-index subgroup of $G/\gamma_c(G)$ and $H$ must be a finite-index subgroup of $G$.

	\medskip
	
	\emph{Claim 2.} For each $g\in G$, we have $g^M\in H$ for $M := [G:H]!$. Equivalently, for each $g\in G$, there exist $w_g\in S^+$ and $z_g\in \gamma_c(G)$ such that $g^M=w_gz_g$.
	
	\smallskip
	
	By the pigeonhole principle, there exist $0\leq i<j\leq [G:H]$ such that $g^iH=g^jH$, and hence $g^{j-i}\in H$. Then $g^M=(g^{j-i})^{\frac M{j-i}}\in H$, which proves Claim 2.
	
	\medskip
	
	\emph{Claim 3.} For each homomorphism $\phi\colon \gamma_c(G)\to\R$, either $\phi\equiv 0$, or there exist $g\in G$ and $h\in \gamma_{c-1}(G)$ such that $\phi([g,h]) < 0$.
	
	\smallskip
	
	\noindent Observe that $\{[g,h] \mid  g\in G, h\in \gamma_{c-1}(G)\}$ is a \emph{symmetric} generating set of $\gamma_c(G)$, and hence either $\phi([g,h])=0$ for all $g,h$ and $\phi\equiv 0$, or the desired $g,h$ exist.
	
	\medskip
	
	\emph{Claim 4.} There exists $t\in S^+\cap \gamma_c(G)$ such that $\phi(t)<0$.
	
	\smallskip
	
	Take $g\in G$ and $h\in\gamma_{c-1}(G)$ as in Claim 3. Since $[g,h] \in \gamma_c(G)\leq\mathcal{Z}(G)$,we have $[g^p,h^q]=[g,h]^{pq}$ for all $p,q\in \Z$. In particular, for any $p\in\Z_{>0}$, we have
	\[ [g,h]^{M^2p^2}= [g^{Mp},h^{Mp}] = w_g^pw_h^pw_{g^{-1}}^pw_{h^{-1}}^p \cdot\, \big(z_gz_hz_{g^{-1}}z_{h^{-1}}\big)^p \] for some $w_g,z_g,w_h,z_h,w_{g^{-1}},z_{g^{-1}},w_{h^{-1}},w_{h^{-1}}$ as in Claim 2. It follows that
	\[ \phi\left( w_g^pw_h^pw_{g^{-1}}^pw_{h^{-1}}^p \right) = M^2p^2 \cdot \phi\big([g,h]\big) - p \cdot \phi \big(z_gz_hz_{g^{-1}}z_{h^{-1}}\big) < 0 \]
	for $p$ large enough. The desired element is $t=w_g^pw_h^pw_{g^{-1}}^pw_{h^{-1}}^p\in S^+\cap\gamma_c(G)$.
	
	
	
	\medskip
	
	(ii) Proof that $S^+$ is a finite-index subgroup of $G$.
	
	\smallskip
	
	First we show that $S^+$ is a subgroup of $G$. For each $s\in S$, there exists $w\in S^+$ such that $sw\in \gamma_c(G)$: this follows from the fact that $s$ must have an inverse $wz$ in $H=S^+\gamma_c(G)$ (with $w\in S^+$ and $z\in\gamma_c(G)$) by Claim 1. So $sw\in S^+\cap \gamma_c(G)$ which is a subgroup by (i), and thus we deduce that $(sw)^{-1}\in S^+ \cap \gamma_c(G)$ which implies $s^{-1}=w(sw)^{-1}\in S^+$.
	
	Finally, we show that $S^+$ has finite-index in $G$: 
	\begin{align*}
		\big[G:S^+\big] & = [G: S^+\gamma_c(G)] \cdot [S^+\gamma_c(G):S^+] \\
		& = \left[G/\gamma_c(G) \,:\, S^+\!/\big(S^+\cap \gamma_c(G)\big)\right] \cdot \big[\gamma_c(G) : S^+\cap\gamma_c(G)\big]<\infty,
	\end{align*}
	using the isomorphism theorem $HK/K\simeq H/H\cap K$ for any $K\le N_G(H)$. 
	This concludes the proof that (1, 2, 3) $\Rightarrow$ (4), and the proof of the proposition.
\end{proof}


\section{The Identity and Subgroup Problems in Nilpotent Groups (Proof of Theorem~\ref{t:main})} \label{s:nilpotent_groups}


In this section, we give an algorithm establishing Theorem~\ref{t:main}, which can be restated as:

\begin{Theorem*} 
	The Identity Problem and the Subgroup Problem, stated below, are uniformly decidable in the class of finitely generated nilpotent groups. 
\medbreak
\noindent The Uniform Identity Problem is the following decision problem:
\begin{itemize}[leftmargin=16mm]
	\item[{\normalfont Input:\hspace*{2.5mm}}] A finite presentation $G=\langle X\mid R\rangle$ and a finite set
$S\subset (X\cup X^{-1})^*$.  
	\item[{\normalfont Output:}]  YES or NO to the question \emph{`Is the identity element $e_G$ in $S^+$?'} 
\end{itemize}
The Uniform Subgroup Problem is the following decision problem:
\begin{itemize}[leftmargin=16mm]
	\item[{\normalfont Input:\hspace*{2.5mm}}] 
	A finite presentation $G=\langle X\mid R\rangle$ and a finite set
$S\subset (X\cup X^{-1})^*$. 
	\item[{\normalfont Output:}] YES or NO to the question \emph{`Is the subsemigroup $S^+$ a subgroup of $G$?'}.
\end{itemize}
\end{Theorem*}
%

The proof of Theorem~\ref{t:main} consists of an algorithm, solving both problems, that relies on the equivalence of parts (1) and (4) of Proposition~\ref{p:discrete_Chow}.
Although decidability of the Subgroup Problem implies decidability of the Identity Problem (the subsemigroup generated by a set $S$ contains the identity if and only if the subsemigroup generated by some non-empty finite subset of $S$ is a subgroup, so we can apply the Subgroup Problem algorithm for all subsets of $S$), we mention both in this section separately, as an alternative way of solving the Identity Problem (but not the Subgroup Problem) is to appeal to Corollary~\ref{crl:easy_out} in step 4(i) of the algorithm, without any mention of Proposition~\ref{p:discrete_Chow} anywhere in the algorithm.

For convenience of exposition, we postpone the proofs of two key facts (Lemmas~\ref{lemma:pos} and~\ref{lemma:neg}) appearing in the proof of Theorem~\ref{t:main}.

\begin{proof}[Proof of Theorem~\ref{t:main}]
We give an algorithm below that solves the Identity and the Subgroup Problems for the finitely presented nilpotent group $G=\langle X\mid R\rangle$. The two problems are solved independently, and the algorithm terminates as soon as it has returned an answer for the problem of interest. 

\begin{enumerate}[leftmargin=8mm, label=(\arabic*)]
\setcounter{enumi}{-1}
\item Set $t:=0$. Let $G_0:=G$, $S_0:=S$, $X_0:=X$ and $R_0:=R$.

\item 

		Compute the torsion-free part of the abelianisation of $G_t$. More precisely, compute the composition $\pi_t$ of the natural projection and a monomorphism
		\[ \pi_t\colon G_t\longto  G_t/I_{G_t}([G_t,G_t]) \longinto \Q^{r_t}, \]
		where the image has full rank. Suppose that $G_t=\langle X_t\mid R_t\rangle$. Consider the canonical homomorphism $\tau_t\colon F(X_t) \to \Q^{X_t}$, where $F(X_t)$ is the free group on $X_t$ and $\Q^{X_t}$ is the $\Q$-vector space over $X_t$. The map $\tau_t$ sends each $x \in X_t$ to the vector with $1$ in the $x$th coordinate, and $0$ otherwise. 
		\begin{enumerate}[leftmargin=6mm, label=(\roman*)]
			\item Compute a basis $B_t$ for the subspace $\langle\tau_t(R_t)\rangle$ generated by the set $\tau_t(R_t)$.
			\item Complete $B_t$ to a basis of $\Q^{X_t}$, adding $x_1,\ldots, x_{r_t}\in X_t$.
			\item Let $\Q^{r_t}:=\bigoplus_{i=1}^{r_t} \Q x_i$, and let $\pi_t$ be the composition of $\tau_t$ with the projection parallel to the subspace $\langle\tau_t(R_t)\rangle$.
		\end{enumerate}

		\item If $r_t=0$ (that is, if $G_t$ is finite) conclude directly:
			\begin{enumerate}[leftmargin=5mm, label=(\roman*)]
				\item If $S_t\ne\emptyset$,
				
				$\quad$ Return YES to both problems.
				
				\item If $S_t=\emptyset$,
				
				$\quad$ Return NO to both problems.
			\end{enumerate}
		\item If $r_t \neq 0$ compute $P_t:=\Conv(\pi(S_t))$. 
		
		$P_t$ is given via a $V$-representation. Find the $H$-representation of $P_t$, that is, express $P_t$ as the intersection of half-spaces $f_{i,t}(\mathbf x)\ge a_{i,t}$ with $f_{i,t}\colon \Q\Z^{r_t}\to \Q\Z$ non-zero linear forms and $a_{i,t}\in \Q\Z$. (See~\cite[\S 1.2]{Ziegler}.)
		
				\item Check the signs of each $a_i$.
		\begin{enumerate}[leftmargin=5mm, label=(\roman*)]
			\item If $a_{i,t}<0=f_{i,t}(\mathbf 0)$ for all $i$, this means that $\mathbf 0$ lies in the interior of $P_t$, that is, there exists $\varepsilon_t>0$ such that $B(\mathbf 0,\varepsilon_t)\subset P_t$. Then $e_G\in S_t^+ \subseteq S^+$ by Proposition~\ref{p:discrete_Chow} (1) (or alternatively, by Corollary~\ref{crl:easy_out}).
			
			\quad Return YES to both problems. 
			
			\item If there exists some $a_{i,t}>0$, then $e_G\notin S_t^+$, by Lemma~\ref{lemma:pos}.
			
			\quad Return NO to both problems.

			\item Otherwise $a_{i,t} \leq 0$ for all $1\leq i\leq r_t$. Let
			\[ G_{t+1}:=\bigcap_{i:a_{i,t}=0} \ker(f_{i,t}\circ\pi_t) \quad \text{ and } \quad S_{t+1} := S_t \cap G_{t+1}. \]
			If $S_{t+1}\subsetneq S_t$,
			
			\quad Return NO to the Subgroup Problem. 
			
			Go to step (1) with new inputs $G_{t+1}=\langle X_{t+1}\mid R_{t+1}\rangle$ and $S_{t+1}$, and increase $t$ by $1$.
			
			Note that a presentation for $G_{t+1}$ can be effectively computed combining the solutions for problems (III) and (IV) of~\cite[p.~5427]{subgroup_pres}. Moreover, $e_G\in S_t^+$ if and only if $e_G\in S_{t+1}^+$, by Lemma~\ref{lemma:neg}. 
		\end{enumerate}
	\end{enumerate}
	It remains to justify that the algorithm terminates. 
The key observation is that the Hirsch length $h(\cdot)$ decreases with each iteration of (4). (Recall that the \emph{Hirsch length} 
 of a nilpotent group $N$ is the number of infinite cyclic factors in any polycyclic series $\{e\}=N_0 \lhd N_1 \dots \lhd N_n=N$, with $N_{i+1}/N_i$ cyclic, of $N$.) Indeed, if there exists $i$ such that $a_{i,t}=0$ (case 4(iii)), then $G_{t+1}\le \ker(f_{i,t}\circ\pi_t)$ and hence
	\[ h(G_{t+1})\le h(\ker(f_{i,t}\circ \pi_t)) = h(G_t)-h(\Z) = h(G_t)-1. \]
	It follows that the algorithm stops in at most $h(G)$ loops; $t$ never exceeds $h(G)$.
\end{proof}


We now prove the lemmas used in the proof of Theorem~\ref{t:main}.

\begin{Lemma}\label{lemma:pos}
Let $G$ be a finitely generated group, $\pi\colon G\onto G/I_G([G,G])\into\Q^r\simeq\Z^r$ the projection to the torsion-free part of the abelianisation of $G$, and $S\subseteq G$. 

Let $P := \Conv(\pi(S))$ and suppose that $P$ is the intersection of half-spaces $f_i(\mathbf x)\ge a_i$ with $f_i\colon\Q \Z^r\to \Q\Z$ non-zero linear forms and $a_i\in \Q\Z$. 

If there exists some $a_i>0$, then $e_G\notin S^+$.
\end{Lemma}

\begin{proof}
If $a_i>0$ for some $1\leq i \leq r$, then $f_i(\pi(s)) \geq a_i>0$, so \[ f_i(\pi(s_1s_2\ldots s_\ell))=\sum_{j=1}^\ell f_i(\pi(s_j))\ge \ell a_i>0 = f_i(\pi(e_G)) \]
			for any word $s_1s_2\ldots s_\ell\in S^+$. 
			Thus $e_G\notin S^+$ since $f_i(\pi(e_G))=0$.
\end{proof}

\newcommand{\new}{\mathrm{new}}
\begin{Lemma}\label{lemma:neg}
Let $\Conv(\pi(S))$, $f_i$, and $a_i$ be as in Lemma~\ref{lemma:pos}, with $a_i \leq 0$ for all $1\leq i\leq r$, and let
\[
G_\new := \bigcap_{i:a_i=0} \ker(f_i\circ\pi) \quad \text{ and } \quad S_\new := S \cap G_\new.
\]
The elements of $S\setminus S_\new$ are not invertible in $S^+$. In particular,
\begin{enumerate}[leftmargin=8mm, label={\normalfont(\arabic*)}]
	\item $S^+$ is not a group if $S\setminus S_\new\ne\emptyset$.
	\item $e_G\in S^+$ if and only if $e_G\in S_\new^+$.
\end{enumerate}
\end{Lemma}
\begin{proof}
Let $I:=\{i \mid a_i=0\}$ and pick $t\in S\setminus S_\new$. Since $t\notin \bigcap_{i\in I} \ker(f_i\circ\pi)$ there exists $j\in I$ such that $f_j(\pi(t))\neq 0=a_j$. So $f_j(\pi(t))>0$ while $f_j(\pi(s))\ge 0$ for all $s\in S$. In particular, for any word $w\in S^+$ containing $t$,
\[ f_j(\pi(w))\ge f_j(\pi(t))>0=f_j(\pi(e_G)) \]
so that $w\ne e_G$, proving that $t$ is not invertible. This justifies (1) and the ``only if" direction of (2). The other implication in (2) is clear, since $S_\new^+ \subseteq S^+$. \qedhere

\end{proof}

\medbreak

\section{Complexity of the Identity and Subgroups Problems and an alternative approach}\label{sec:complexity}


The complexity of the algorithm for the Identity and Subgroup Problems (cf. Proof of Theorem~\ref{t:main}) is not always polynomial, since in Step (1) the basis for the torsion-free abelianisation of $G_t$, at each step $t$, can be exponential in the size of $X_t$ (see for example~\cite{HavasWagner98}). This leads to the representations of words in $S_t$ being potentially exponential over the presentation of the input group; hence, although our algorithm has a polynomial number of steps, in many of the steps we may deal with operands with very large values. Step (1) is the main culprit for an exponential explosion, as in Step (3) we can find an $H$-representation for $\Conv(\pi(S_t))$ in $O(r_t n f)$ operations, where $n=\abs{S_t}$ is the number of vertices and $f$ is the number of facets~\cite{facet_enumeration}. Moreover, $f=O(n^{\lfloor r_t/2\rfloor})$ by the McMullen Upper Bound Theorem (see~\cite[\S 6.4]{Ziegler}).

We therefore suggest another algorithm, based on Dong's~\cite[Algorithm 1]{DongSODA2024}, which can run in polynomial time if a unitary matrix representation for the largest torsion-free quotient of the nilpotent group is given. We assume that for a fixed $G$, finding a unitary matrix representation is part of the preprocessing, so we treat this as an additive constant (albeit a significant one).

\subsection{Malcev completions and Lie algebras.} We start with some preliminaries on nilpotent groups and Lie algebras (see~\cite[Chapter 6]{NilpotentGpBook} and~\cite{Stewart70} for background).

\medskip

For any torsion-free nilpotent group $N$ we can define the \emph{rational Malcev completion} $N_\Q$, or \emph{$\Q$-Malcev completion of $N$}, as the unique torsion-free nilpotent group satisfying (1) $N$ embeds into $N_\Q$, (2) for all $g \in N$ and $k\geq 0$, there exists a unique $h\in N_\Q$ such that $h^k=g$, and (3) for all $h \in N_\Q$, there exists a $k \geq 0$ such that $h^k \in N$. Furthermore, we can associate to every torsion-free nilpotent group $N$ the corresponding Lie algebra $\mathfrak{n}$ via the bijective map $\log\colon N \longto \nk$. This Lie algebra can be realised explicitly by embedding $N$ into the unitary triangular group $\mathrm{UT}(m,\Q)$ of appropriate dimension $m>0$ and working in the corresponding Lie subalgebra  $\nk$ within the algebra $\mathfrak{u}(m):=\log(\mathrm{UT}(m, \Q))$.

\medskip

For a set $T$ of matrices (in a Lie algebra of nilpotency class $c$), we define 
\begin{align*}
\Lie(T) & := \Vect_\Q\big\{[t_{1},\ldots,t_l]_\nk : t_i\in T,\; l\ge 1 \} \quad\text{and} \\
[T,T]_\nk & := \Vect_\Q\big\{[t_{1},\ldots,t_l]_\nk : t_i\in T,\; l\ge 2 \big\},
\end{align*}
where $[t]_\nk:= t$ and $[t_{1},\ldots,t_l]_\nk := [[t_1, \ldots, t_{l-1}]_\nk, t_l]_\nk$. Note that all brackets of length $l\geq c+1$ are equal to $\textbf{0}$, so both spaces are finitely generated as soon as $T$ is finite. The space $\Lie(T)$ is the subalgebra of $\nk$ generated by $T$, and $[T,T]_\nk$ is the ideal of $\Lie(T)$ generated by commutators $[t_1,t_2]_\nk$. 

\begin{Lemma}[{\cite[Lemma 2.5.2]{Stewart70}}] \label{lem:LieAlg}
	Let $N$ be a torsion-free nilpotent group. We have
	\[ I_N([N,N]) = [N_\Q,N_\Q] \cap N  \quad\text{and}\quad \log[N_\Q,N_\Q] = [\nk,\nk]_{\nk}. \]
	In particular, the following diagram commutes:
	\begin{center}
		\begin{tikzcd}
			N \arrow[r, hook] \arrow[d, two heads] & N_\Q \arrow[r, "\sim", "\log"'] \arrow[d, two heads] & \nk \arrow[d, two heads] \\
			N / I_N[N,N] \arrow[r, hook] & N_\Q/[N_\Q,N_\Q] \arrow[r, "\sim", "\log"'] & \nk/[\nk,\nk]_\nk
		\end{tikzcd}
	\end{center}
\end{Lemma}

Finally, a \emph{convex cone} or simply \emph{cone} of a vector space over an ordered field or ring $F$ (in our case $F=\mathbb{Z}, \mathbb{Q}, \mathbb{R}$), is a set $C$ such that $aC=C$ and $C+C=C$, for any positive scalar $a \in F$. 
\begin{Example}
Suppose that $V$ is the set of rational solutions to a system of linear homogeneous equations with $k$ variables. Then $V$ is a vector field over $\mathbb{Q}$, and the set $V\cap \mathbb{Z}_{\geq 0}^k$ of the vectors in $V$ with positive integer coefficients is a cone in $V$.
\end{Example}

\subsection{Alternative algorithm for the Identity and Subgroup Problems} \label{sec:DongAlgo}  \ 

\medskip

\noindent For a nilpotent group $G$, let $T$ be the finite subgroup consisting of all torsion elements of $G$ and let $\overline G=G/T$ be the largest torsion-free quotient of $G$.

\medskip

\begin{enumerate}[leftmargin=8mm, label=(\arabic*)]
		\setcounter{enumi}{-1}
		\item Identify $S_0:=S$ with a subset of matrices in the Lie algebra $\mathfrak{g}$ of the $\Q$-Malcev completion $\overline G_\Q$. (That is, although $S_0:=\log(S)$, we drop the $\log$ notation.)
		
		\medskip
		
		\item For $t \geq 0$, suppose that $S_t=\{s_1,s_2,\ldots,s_\ell\} \subset \mathfrak{g}$. 
		
		\smallskip
		
		\noindent Let $C$ be the cone of solutions $\textbf{x}=(x_1, \dots, x_\ell)\in\Q_{\ge 0}^\ell$ to
		\[ x_1s_1+x_2 s_2+\ldots+x_\ell s_\ell = \textbf{0} \mod{[S_t,S_t]_\gk}, \]
		and compute $S_{t+1}:=\{s_i \in S_t \mid \exists\textbf{x}\in C\text{ with }x_i>0\}$.
		
		
		\medskip
		
		\item If $S_{t+1}=S_t$, 
			
			Return YES to the Identity Problem. ($\star$). 
			
		\noindent If $S_{t+1}=\emptyset$, 
		
			Return NO to the Identity Problem.
			
		\noindent Else, go back to step (1) with input $S_{t+1}$. 
	\end{enumerate}
	For the Subgroup Problem, step (2) should be replaced with
	\begin{itemize}[leftmargin=10mm]
		\item[(2')] If $S_1=S_0\ne \emptyset$,
			
			Return YES to the Subgroup Problem ($\star$'). 
			
		\noindent Else,
			
			Return NO to the Subgroup Problem
	\end{itemize}
	\begin{Proposition}\label{p:complexity}
	
	Fix a group $G$ of nilpotency class $c$, with a finite Malcev basis $B$. Let $S$ be a set of elements written as words in Malcev normal form over $B$. Let $n$ be the total bit length of exponents of the words in $S$. The algorithm (\ref{sec:DongAlgo}) solves the Identity and Subgroup Problems in time polynomial in $n$. More precisely, the time complexity is bounded by $O(\abs{S}^{4.5c+2}n^2\log(n))$.
	\end{Proposition}

	We make several remarks on complexity before proving the proposition, and emphasize the difference between solving the problems for a fixed group versus solving the problems uniformly, when the group presentation is included in the input.
	\begin{Remark} \label{rem:complexity}
 $\,$

\textbf{1.} We assume certain preprocessing steps in the algorithm have constant time complexity for a given group $G$. For example, we assume that we are given a representation of $\overline{G}$ as unitary matrices, and as such, we do not include finding the matrix representation as part of the  complexity. 


\textbf{2.} Regarding a uniform algorithm for all nilpotent groups, our approach and the alternative one by Shafrir do not provide polynomial algorithms. 

First, the algorithm (\ref{sec:DongAlgo}) runs in polynomial time if we are given a representation of $\overline{G}$ as unitary matrices. It is shown in~\cite[Theorem 3.9]{GulWeiss17} that in the main algorithms known for finding matrix representations, given a Malcev basis for a torsion-free group it is possible to have an exponential explosion, in terms of the size of the Malcev basis, when computing the dimension of the matrix representation for the group.  Polynomial bounds for the dimension of the matrix representation are available once the nilpotency class is fixed, and there is an algorithm finding this representation~\cite[Theorem 5.2]{GulWeiss17}.
	
	
	

		\end{Remark}
	\begin{proof}
	We assume that we are given the embedding of $\overline{G}$ into the unitary triangular group $\mathrm{UT}(m, \Q)$ of appropriate dimension $m>0$ and work within the corresponding Lie algebra $\mathfrak{u}(m)=\log(\mathrm{UT}(m, \Q))$. 
	During the algorithm we work with the elements of $S$ as integer matrices (with entries represented/stored in binary) whose sizes differ from those of the words in $S$: the entry in each matrix is a polynomial in the Malcev form exponents, 
	hence the overall input size (of elements in $S$ as matrices) is only changed by a multiplicative constant compared to the initial input length (because the logarithm of a fixed polynomial $p(x)$ is a constant multiple $\log(x)$).
	
	First, note that the algorithm will terminate in at most $h(G)$ iterations, since either $S_{t+1}=S_t$ and the algorithm terminates, or 
	\[ \dim_\Q\Lie(S_{t+1})<\dim_\Q\Lie(S_t). \]
	
	(i) Correctness.
	
	\smallbreak
	
	Regarding (0), note that we can work in the torsion-free quotient $\overline{G}$ of $G$ en lieu of $G$, as $S^+$ contains $e_G$ (resp.\ is a subgroup of $G$) if and only if $\overline{S}^+$ contains $e_{\bar G}$ (resp.\ is a subgroup of $\overline{G}$) by~\cite[Lemma 3.2.6]{DongThesis} (resp.~\cite[Lemma 3.2.5]{DongThesis}).
	
	
	\medbreak
	
	Regarding (2), the difficult part is marked ($\star$). Recall that Lemma~\ref{lem:LieAlg} gives
	\begin{center}
		\begin{tikzcd}
			\langle S_t\rangle \arrow[r] \arrow[d, swap, two heads, "\pi_t"] & \Lie(S_t) \arrow[d, two heads] \\
			\langle S_t\rangle /I([S_t,S_t]) \arrow[r] & \Lie(S_t)/[S_t,S_t]_\gk
		\end{tikzcd}
	\end{center}
	The condition $S_{t+1}=S_t$ means that, in $\Lie(S_t) / [S_t,S_t]_\gk$, we can write $\textbf{0}$ as a positive linear combination of \emph{all the} $s_i$, hence $\textbf{0}$ sits in the interior of the convex hull $\Conv(\pi_t(S_i))$, and we may apply Proposition~\ref{p:discrete_Chow} and conclude that $S_t^+$ is a subgroup, and $e_G\in S_t^+$. On the other hand, if there exists $s\in S_t\setminus S_{t+1}$, we can conclude that $\pi_t(s)$ is not invertible in $\pi_t(S_t^+)$ (hence $s$ is not invertible in $S_t^+$). It follows that $S_t^+$ is not a group. Moreover, any word evaluating to $e_G$ cannot include any such letter $s$, hence $e_G\in S_t^+$ if and only if $e_G\in S_{t+1}^+$.
	
	\medbreak
	
	For the same reasons, if $S_0=S_1$, then $S^+=S_0^+$ is a subgroup. Otherwise, $S^+=S_0^+$ contains non-invertible elements (any element in $S_0\setminus S_1$), hence $S^+$ cannot be a subgroup. This justifies the step (2') and in particular part ($\star$') .
	
	
	\medbreak
	
	(ii) Complexity.
	As we mention in Remark~\ref{rem:complexity}(1.), we assume that finding a matrix representation is part of the preprocessing, and does not count towards the complexity (other than as a constant).
	\smallbreak
	The only significant step in the algorithm is part (1), which by~\cite[Lemma 3.4]{DongSODA2024} can be performed in polynomial time on the size $n$ of the input (as matrices), as it relies on finding rational solutions to a linear programming problem (see~\cite[Lemma 3.3.4]{DongThesis}). More precisely, we have to solve $\abs S$ systems over $O(\abs S^c)$ variables (for each vector $s\in S_t$ and for each commutator generating $[S_t,S_t]_\gk$), with $O(\abs S)$ linear constraints, and total bit length $O(\abs S^cn)$. Using Karmarkar's algorithm~\cite{Linear_programming}, each system can be solved in $O(\abs S^{4.5c+1}n)$ operations with $O(n)$-digits numbers, yielding a total time of $O(\abs S^{4.5c+2}n^2 \log n)$ \cite{Interger_Multiplication}. 
	\end{proof}


\section{Fixed-target submonoid membership}\label{sec:fixed-target}

	
	As mentioned in the introduction, the decidability of the Identity Problem for every finitely generated nilpotent group stands in contrast to the undecidability of the Submonoid Membership Problem for certain members of this class~\cite{romankov2022undecidability}. In this section, we show that an intermediate problem -- the \emph{Fixed-Target Submonoid Membership Problem} -- is also undecidable for certain finitely generated nilpotent groups. This problem was recently introduced by Gray and Nyberg-Brodda~\cite{gray2024membership}. Our result follows from recent work of Pak and Soukup on Green functions~\cite{pak_soukup}.
	\begin{Proposition} \label{fixed_target}
		There exists a finitely generated nilpotent group $G=\langle X\rangle$ and a fixed element $g_0\in G$ such that the following problem is undecidable:
		\begin{itemize}[leftmargin=16mm]
			\item[{\normalfont Input:\hspace*{2.5mm}}] 
			A finite set $S\subset (X\cup X^{-1})^*$. 
			\item[{\normalfont Output:}] YES or NO to the question \emph{`$g_0\in S^*$?'}.
		\end{itemize}
	\end{Proposition}
	\begin{Lemma}[{\cite[Lemma 4.4]{pak_soukup}}]\label{lem:pak}
		Let $f\in\Z[x_1,\ldots,x_k]$ be a polynomial. There exist a positive integer $m=m(\deg(f),k)$ and effectively computable matrices $P,Q,A_1,\ldots,A_k\in \mathrm{UT}(m,\Z)$ such that the equation
		\begin{equation*} \label{eq:special_zero_word}
			P\cdot W_1\cdot Q\cdot W_2\cdot P^{-1}\cdot W_3\cdot Q^{-1} \cdot W_4 = e
		\end{equation*}
		admits a solution $W_1,W_2,W_3,W_4\in \langle A_i \mid i=1, \ldots, k \rangle$ if and only if the Diophantine equation $f(x_1,\ldots,x_k)=0$ admits an integer solution.
	\end{Lemma}
	\begin{proof}[Proof of Proposition~\ref*{fixed_target}]
		We fix integers $d, k > 0$ such that it is undecidable if a Diophantine equation of degree $d$ over $k$ variables admits an integer solution, and use $m=m(d,k)$ (provided by Lemma~\ref{lem:pak}) as the corresponding matrix dimension as follows. Let $G := \mathrm{UT}(m,\Z)\times \mathrm{UT}(3,\Z)$ with
		\[ \mathrm{UT}(3,\Z)= \langle x,y \;\big|\; [x,[x,y]]=[y,[x,y]]=e \rangle. \]
		For a polynomial $f\in\R[x_1,\ldots,x_k]$ of degree $d$, we consider the following pairs of matrices in $G$ (with $P$, $Q$, $A_i\in\mathrm{UT}(m,\Z)$ given by Lemma~\ref{lem:pak}):
			\[
			\begin{array}{lll}
				\tilde P_+=(P,x^2), & \tilde Q_+=(Q,y^2), & \tilde A_{i,+}=(A_i,e), \\
				\tilde P_-=(P^{-1},x^3), & \tilde Q_-=(Q^{-1},y^3), & \tilde A_{i,-}=(A_i^{-1},e).
			\end{array}
			\]
		We take $g_0=(e,\,x^2y^2x^3y^3)$ and $S=\{\tilde P_\pm,\tilde Q_\pm, \tilde A_{i,\pm}\}$. Looking at the second component, we see that any word over $S$ evaluating to $g_0$ must contain $\tilde P_+$, $\tilde Q_+$, $\tilde P_-$ and $\tilde Q_-$ exactly once, in that order. Therefore $g_0\in S^*$ if and only if there exist $W_0,W_1,W_2,W_3,W_4\in \langle A_i \mid i=1, \ldots, k \rangle$ such that
		\begin{align*}
			W_0\cdot P\cdot W_1\cdot Q\cdot W_2\cdot P^{-1}\cdot W_3\cdot Q^{-1} \cdot W_4 & = e \\
			\iff P\cdot W_1\cdot Q\cdot W_2\cdot P^{-1}\cdot W_3\cdot Q^{-1} \cdot W_4W_0 & = e.
		\end{align*}
		In turn, this is equivalent to the Diophantine equation $f(x_1,\ldots,x_k)=0$ admitting an integer solution, which is undecidable.
	\end{proof}
	
	\begin{Remark}
		We would like to highlight the link between the undecidability of the Fixed-Target Submonoid Membership Problem (FTSMP) in a group and the strategy initiated in~\cite{garrabrant2017words, pak_soukup} to construct generating sets for which the Green series of the group is not $D$-finite. It is not clear that the undecidability of FTSMP implies the lack of $D$-finiteness, but there are similarities in how one approaches both problems. In particular, it would be interesting to construct such generating sets for the right-angled Artin group $A(P_4)$, or for Thompson's group $F$. 
	\end{Remark}


\section{Orientation-preserving affine transformations of $\Q$} \label{s:aff}


One may wonder if Proposition~\ref{p:fully-orderable-identity} applies beyond torsion-free nilpotent groups. We first note that it does not apply to the entire class of torsion-free metabelian groups. Indeed, Rhemtulla shows that the fundamental group of the Klein bottle $\langle a,b  \mid ab=ba^{-1} \rangle\simeq \Z\rtimes_{-1}\Z$ is not fully left-orderable \cite[\S 4]{Rhe72}.  The goal of this section is to show that Proposition~\ref{p:fully-orderable-identity} applies to a family of orientation-preserving affine transformations of $\Q$, which are proved to be fully left-orderable.

\medskip

For each $\lambda\in\Q_{>1}$, we consider the groups $G_\lambda\le\mathrm{Aff}_+(\Q)$, defined as
\[ G_\lambda = \left\{ x\mapsto \lambda^nx+c \;\Big|\; n\in\Z,\, c\in\Z\big[\lambda,\lambda^{-1}\big]\right\} \simeq \Z\big[\lambda,\lambda^{-1}\big]\rtimes_\lambda \Z. \]
Each group $G_\lambda$ is generated by the translation $a\colon x \mapsto x+1$ and the homothety $b\colon x\mapsto \lambda x$. If $\lambda=\frac{p}{q}$ then $G_\lambda$ is a homomorphic (but in general not isomorphic) image of the Baumslag-Solitar group $BS(p,q)=\langle a,b \mid b^{-1}a^pb=a^q \rangle$. It is an isomorphic image for $\lambda=m\in\Z_{\ge 2}$, and so the family $G_\lambda$ notably contains the groups $G_m\simeq BS(1,m)$. 

For a fixed $\lambda\in\Q_{>1}$ let $\pi\colon G_\lambda\to\Z$ denote the projection on the second factor, and let $e=e_{G_{\lambda}}\colon x \mapsto x$. We recall the construction of left-orders on $G_\lambda$. 
\begin{itemize}[leftmargin=8mm]
	\item[(1)] For each pair $(\sigma_1,\sigma_2)\in\{+,-\}^2$, we define $f\succ e$ if
	\[ \sigma_1\cdot \pi(f)>0 \quad\text{or}\quad \big(\pi(f)=0\text{ and }\sigma_2\cdot f(0)>0\big).\]
	These four pairs give four orders, all Conradian. 
	\item[(2)] For each $p\in\R$ and pair $(\sigma_1,\sigma_2)\in\{+,-\}^2$, we define an order $f\succ e$ if
	\[ \sigma_1 \cdot (f(p)-p) > 0 \quad\text{or}\quad \big(f(p)=p \text{ and }\sigma_2\cdot \pi(f)>0\big).\]
	Note that, for $p\in\R\setminus\Q$, the only map $f\in G_\lambda$ satisfying $f(p)=p$ is $f=e$, hence we get the same order for $\sigma_2=+$ or $\sigma_2=-$. Otherwise the orders described in (2) are pairwise distinct, and none are Conradian.
\end{itemize}

Rivas proved that these orders are the only left-orders on the Baumslag-Solitar group $BS(1,2)$ \cite[Theorem 4.2]{Rivas_BS12}, and a complete description for $BS(1,m)$ can be found in \cite[Section 1.2.2]{DeroinNavasRivas}. We provide an alternate proof of this result, which moreover implies that all $G_\lambda$ are fully left-orderable.
\begin{Lemma}\label{lem:classification_of_orders}
	Let $S\subseteq G_\lambda$ and suppose that no left-order $\succ$ in the previous list satisfies $s\succ e$ for all $s\in S$. Then $e \in S^+$.
\end{Lemma}

	For each affine transformation $f \colon x\mapsto \lambda^nx+c \in G_\lambda$, we define
	\begin{align*}
		F_f & := \{p\in\R \mid f(p)=p \}, \\
		L_f & := \{p\in\R \mid f(p)<p \}, \\
		R_f & := \{p\in\R \mid f(p)>p \}.
	\end{align*}
	We decompose $G_\lambda\setminus\{e\}$ into three classes, as follows. 
	\begin{itemize}[leftmargin=8mm]
		\item $f$ is \emph{expanding} if $n=\pi(f)>0$. In this case
		\[ F_f=\{p_f\},\quad L_f = (-\infty,p_f),\quad\text{and}\quad R_f=(p_f,+\infty), \]
		where $p_f=\frac{c}{1-\lambda^n}$ is the unique fixed point of $f$.
		\item $f$ is \emph{contracting} if $n=\pi(f)<0$. In this case
		\[ F_f=\{p_f\},\quad L_f = (p_f,+\infty),\quad\text{and}\quad R_f=(-\infty,p_f). \]
		\item $f$ is a \emph{translation} if $\pi(f)=0$. We have $R_f=\R$ if $c>0$, or $L_f=\R$ if $c<0$.
	\end{itemize}
\begin{center}
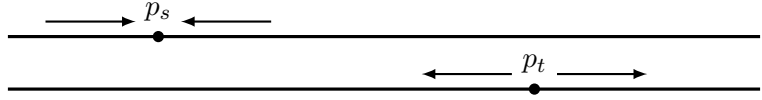

	\begin{tikzpicture}
		\draw[very thick] (-5,0) -- (5,0);
		\node[circle, inner sep=1.5pt, fill=black, label=above:$p_s$] at (-3,0) {};
		\draw[thick, latex-] (-3.3,.2) -- (-4.5,.2);
		\draw[thick, latex-] (-2.7,.2) -- (-1.5,.2);
	\end{tikzpicture}
	
	\begin{tikzpicture}
		\draw[very thick] (-5,0) -- (5,0);
		\node[circle, inner sep=1.5pt, fill=black, label=above:$p_t$] at (2,0) {};
		\draw[thick, latex-] (.5,.2) -- (1.7,.2);
		\draw[thick, latex-] (3.5,.2) -- (2.3,.2);
	\end{tikzpicture}
	\captionsetup{margin=7mm, font=small}
	\captionof{figure}{A contracting map $s$ and an expanding map $t$.}
\end{center}
\begin{proof}[Proof of Lemma~\ref{lem:classification_of_orders}]
	We first consider type (2) orders. By assumption, none of them satisfies $s\succ e$ for each $s \in S$. 
	For each $p\in\mathbb R$, this means one of two things: either (i) there exist $s,t\in S$ such that $s(p)=t(p)=p$ and $\pi(s)\le 0\le \pi(t)$, or (ii) there exist $s,t\in S$ such that $s(p)<p<t(p)$. We consider the stabilizer
	\[ \Stab(p)=\{f\in G_\lambda:f(p)=p\}. \]
	Observe that $\pi\colon \Stab(p)\to\Z$ is injective, hence $\Stab(p)$ is cyclic (possibly trivial if $p\in\R\setminus\Q$). In case (i), we have $s,t\in S\cap\Stab(p)$ with $\pi(s)\le 0\le \pi(t)$ and it follows that $e\in S^+$. From now on, we suppose that we are in the second case (ii) for each $p\in\R$. Equivalently, we suppose that
	\begin{equation} \label{eq:everyone_moves} \tag{$\star$}
		\bigcup_{f\in S} L_f = \bigcup_{f\in S}R_f = \R.
	\end{equation}
	
	Now consider type (1) orders. By assumption, no order satisfies $s\succ e$ for all $s\in S$. Once again, this means one of two things. In the first case, there exist $s,t\in S$ such that $\pi(s)=\pi(t)=0$ and $s(0)\le 0\le t(0)$ and therefore $e\in S^+$ (since $\ker\pi\le (\Q,+)$ is locally cyclic). We suppose that we are in the second case
	\begin{equation}  \label{eq:contract_and_expand} \tag{$*$}
		\exists s,t\in S,\quad \pi(s)<0<\pi(t).
	\end{equation}
	Without loss of generality (up to the outer automorphism $f\mapsto -f(-\,\cdot\,)$), we may assume $p_s\le p_t$. We highlight translations in both directions in $S^+$.
	\begin{itemize}[leftmargin=8mm]
		\item Using (\ref{eq:contract_and_expand}), there exist $m,n\in\Z_{>0}$ such that $\pi(s^mt^n)=m\pi(s)+n\pi(t)=0$. Moreover, $s^mt^n(p_t)=s^m(p_t)\le p_t$. We conclude that $s^mt^n\in S^+$ is a translation to the left (possibly the identity).
		\item In order to satisfy (\ref{eq:everyone_moves}), we still need
		\[ [p_s,p_t] \subset \bigcup_{f\in S} R_f. \]
		Observe that $[p_s,p_t]$ is compact, the sets $R_f$ are open, and their poset (with respect to inclusion) is the union of two chains. Therefore $[p_s,p_t]$ is covered by at most two $R_f$ sets (with $f\in S$). We have four cases:
		\begin{itemize}[leftmargin=5mm]
			\item There exists $u\in S$ such that $R_u=\R$, then $u$ is our translation to the right.
			\item There exists $u\in S$ such that $R_u=(p_u,+\infty)$ with $p_u< p_s$. Using an argument as above, we find a translation to the right of the form $s^mu^n$.
			\item There exists $v\in S$ such that $R_v=(-\infty,p_v)$ with $p_t< p_v$. Using an argument as above, we find a translation to the right of the form $t^mv^n$.
			\item There exist $u,v\in S$ such that $R_u=(p_u,+\infty)$ and $R_v=(-\infty,p_v)$ with $p_u< p_v$. Then we find a translation to the right of the form $u^mv^n$.
		\end{itemize}
	\end{itemize}
	This proves that $S^+$ contains translations in both directions, hence $e \in S^+$ (as $\ker\pi=\Z[\lambda,\lambda^{-1}]\le \Q$ is locally cyclic).
\end{proof}
\begin{Corollary} \label{cor:affine_fully}
	For every $\lambda\in\Q_{>1}$, the group $G_\lambda$ is fully left-orderable. 
	
	Every partial left-order can be extended to a left-order of type {\normalfont(1)} or {\normalfont(2)}, and so every left-order on $G_\lambda$ is of type {\normalfont(1)} or {\normalfont(2)}.
\end{Corollary}
\begin{proof}
	Consider a partial left-order $\succ$ on $G_\lambda$ and let $S:= \{g\in G_\lambda \mid g\succ e\}$ be its positive cone. Combining Lemma~\ref{lem:classification_of_orders} with the fact that $e\notin S=S^+$, we conclude the existence of a left-order $\bar\succ$ on $G_\lambda$ of type (1) or (2) such that $s\,\bar\succ\, e$ for all $s \in S$, that is, the left-order $\bar\succ$ extends $\succ$.
\end{proof}


The Identity Problem is decidable in metabelian group by \cite{DongSTOC2024}. This implies that the Left Order-Extension Problem is decidable. However, the interesting direction of Proposition~\ref{p:fully-orderable-identity} is really the opposite one: from the explicit classification of left-orders on $G_\lambda$, we deduce an efficient algorithm for the Identity Problem.
\begin{Theorem}
	The Identity Problem, for a finite set $S$ in $G_\lambda$, is decidable in $O(\abs S)$ operations.
\end{Theorem}

\begin{proof} The algorithm for solving the Identity Problem for $S$ is the following.
	\begin{enumerate}[leftmargin=8mm, label=(\arabic*)]
		\item Compute the two sets
		\begin{align*}
			L_S:= \bigcap_{s\in S} L_s = \begin{cases}
				\hspace*{3.5mm} \emptyset \hspace{6mm} \text{if $S$ contains a translation $s(x)=x+c$ with $c\ge 0$,} \\
				\big(\!\sup\{p_s\mid \pi(s)<0\},\;\inf\{p_s\mid \pi(s)>0\}\big) \hspace*{6mm} \text{otherwise,}
				\end{cases} \\
			R_S:=\bigcap_{s\in S} R_s = \begin{cases}
				\hspace*{3.5mm} \emptyset \hspace{6mm} \text{if $S$ contains a translation $s(x)=x+c$ with $c\le 0$,} \\
				\big(\!\sup\{p_s\mid \pi(s)>0\},\;\inf\{p_s\mid \pi(s)<0\}\big) \hspace*{6mm}\text{otherwise,}
			\end{cases}
		\end{align*}
		where $(a,b)=\emptyset$ if $a\ge b$. This takes $O(\abs S)$ operations, as it involves computing the fixed point of each map in $S$, and keeping track of different maxima and minima along the way.
		
		\smallskip
		
		\item If $L_S=R_S=\emptyset$
		
		Return YES to the Identity Problem.
		
		\noindent Otherwise
		
		Return NO to the Identity Problem.
	\end{enumerate}
	The correctness of this algorithm relies on the chain of equivalences
	\[ e\notin S^+ \;\iff\; C_S \ne \emptyset \;\iff\; (L_S\ne \emptyset \text{ or } R_S \ne \emptyset), \]
	where $C_S$ is the set of left-orders $\left\{ \succ\; \in \mathrm{LO}(G_\lambda) \;\big|\; \forall s\in S,\, s\succ e \right\}$ (and $\mathrm{LO}(G_\lambda)$ denotes the set of all left-orders). Corollary~\ref{cor:affine_fully} justifies the first equivalence. 
	
	We prove the second equivalence:
	
	\smallskip
	
	For the right-to-left direction, if $p\in R_S$ (resp.\ $p\in L_S$), then type (2) orders defined with reference point $p$ and $\sigma_1=+1$ (resp.\ $\sigma_1=-1$) belong to $C_S$.
	
	\smallskip
	
	For the left-to-right direction, we suppose that $C_S\ne\emptyset$. From our classification of left-orders or from \cite{Orders_virt_solvable}, we see that $\mathrm{LO}(G_\lambda)$ is a Cantor set. The set $C_S$ is open 
	and non-empty, and therefore is uncountable. In particular, there exists an order $\succ$ of type (2) in $C_S$ with reference point $p\in\R\setminus\Q$. We deduce that $p\in R_S$ if $\sigma_1=+1$, or $p\in L_S$ if $\sigma_1=-1$.
\end{proof}

\bigbreak

We end with a few remarks and directions for future research.
\begin{Question}
	Is $\mathrm{Aff}_+(\Q)$ fully left-orderable?
\end{Question}
Note that $\mathrm{Aff}_+(\R)$ is not fully left-orderable, as the subgroup generated by $a\colon x\mapsto\pi x$ and $b\colon x\mapsto x+1$ is isomorphic to $\Z\wr\Z$, which is not fully left-orderable (take $S=\{a^2,b,(aba)^{-1}\}$). In this case, we may however wonder if each partial left-order $S^+$ can always be extended to a left-order on $\langle S\rangle$.

\smallskip

Another interesting class of groups is $\mathrm{SOL}_T=\Z^2\rtimes_T\Z$ with $T\in \mathrm{SL}_2(\Z)$ a hyperbolic matrix (that is, with trace $>2$). Left-orders on $\mathrm{SOL}$ have been classified \cite[Section 3]{Orders_virt_solvable}. It would be interesting to see if these groups are fully left-orderable and our tools can be used here and turned into an efficient algorithm for the Identity Problem, as this is a non-trivial case in \cite{DongLICS2023}.

\smallskip

In order to answer these questions, it is important to note that one can restrict to finitely generated left-orders of finitely generated subgroups.
	\begin{Proposition} Let $G$ be a group. The following conditions are equivalent:
		\begin{enumerate}[leftmargin=8mm, label={\normalfont(\arabic*)}]
			\item $G$ is fully left-orderable.
			\item Every finitely generated subgroup of $G$ is fully left-orderable.
			\item For every finitely generated subsemigroup $S^+$ of $G$, if $S^+$ is a partial left-order (i.e., $e\notin S^+$), then $S^+$ extends to a left-order.
		\end{enumerate}
	\end{Proposition}
	\begin{proof}
		Using Zorn's lemma, every positive cone is included inside a maximal one. If $G$ is not fully left-orderable, then there exists a maximal positive cone $P$ such that $P\sqcup\{e\}\sqcup P^{-1}\ne G$. In particular, there exists $g\in G\setminus(P\sqcup\{e\}\sqcup P^{-1})$. By maximality, we have $e\in (P\cup\{g\})^+$ and $e\in (P\cup\{g^{-1}\})^+$. In particular, there exists a finite subset $S\subseteq P$ such that $e\in(S\cup\{g\})^+$ and $e\in(S\cup\{g^{-1}\})^+$. Therefore $S^+$ cannot be extended to a total order, despite $e\notin S^+$ (since $S^+\subseteq P$). Moreover, the finitely generated subgroup $\langle S,g\rangle$ is not fully left-orderable.
	\end{proof}

\smallskip

\begin{Question}
	Is the Left-Order Extension Problem decidable for metabelian groups? Is the Bi-Order Extension Problem decidable for metabelian groups?
\end{Question}


\section*{Acknowledgements}
The first author acknowledges support from the Swiss National Science Foundation grant 200020-200400. The second and third named authors acknowledge a Scientific Exchanges grant (number IZSEZ0$\_ $213937) of the Swiss National Science Foundation. The second author would like to thank Universität Bern for their hospitality and support during the initial stages of writing this paper.


\bibliographystyle{plain}

\begin{thebibliography}{10}

\bibitem{AF88}
Marlow~E. Anderson and Todd~H. Feil.
\newblock {\em Lattice-Ordered Groups: An Introduction}.
\newblock Springer, 1988.

\bibitem{facet_enumeration}
David Avis and Komei Fukuda.
\newblock A pivoting algorithm for convex hulls and vertex enumeration of
  arrangements and polyhedra.
\newblock In {\em Proceedings of the Seventh Annual Symposium on Computational
  Geometry}, pages 98--104, 1991.

\bibitem{BHP2017}
Paul~C. Bell, Mika Hirvensalo, and Igor Potapov.
\newblock The identity problem for matrix semigroups in {$\mathrm{SL}_2(\mathbb
  Z)$} is {NP}-complete.
\newblock In {\em Proceedings of the {T}wenty-{E}ighth {A}nnual {ACM}-{SIAM}
  {S}ymposium on {D}iscrete {A}lgorithms}, pages 187--206. SIAM, Philadelphia,
  PA, 2017.

\bibitem{BP2010}
Paul~C. Bell and Igor Potapov.
\newblock On the undecidability of the identity correspondence problem and its
  applications for word and matrix semigroups.
\newblock {\em Internat. J. Found. Comput. Sci.}, 21(6):963--978, 2010.

\bibitem{BG09}
Vasily~V. Bludov and Andrew~M.~W. Glass.
\newblock Word problems, embeddings, and free products of right-ordered groups
  with amalgamated subgroup.
\newblock {\em Proc. London Math. Soc.}, 99(3):585--608, 2009.

\bibitem{BodartDong2024}
Corentin Bodart and Ruiwen Dong.
\newblock The identity problem in virtually solvable matrix groups over
  algebraic numbers, 2024.
\newblock Preprint at \url{https://arxiv.org/abs/2404.02264}.

\bibitem{CR16}
Adam Clay and Dale Rolfsen.
\newblock {\em Ordered Groups and Topology}, volume 176 of {\em Graduate
  Studies in Mathematics}.
\newblock American Mathematical Society, 2016.

\bibitem{NilpotentGpBook}
Anthony~E. Clement, Stephen Majewicz, and Marcos Zyman.
\newblock {\em The Theory of Nilpotent Groups}.
\newblock Birkh\"auser/Springer, Cham, 2017.

\bibitem{CM19}
Almudena Colacito and George Metcalfe.
\newblock Ordering groups and validity in lattice-ordered groups.
\newblock {\em J. Pure Appl. Algebra}, 223(12):5163--5175, 2019.

\bibitem{conrad}
Paul Conrad.
\newblock Right-ordered groups.
\newblock {\em Michigan Mathematical Journal}, 6:267--275, 1959.

\bibitem{DeroinNavasRivas}
Bertrand Deroin, Andrès Navas, and Cristóbal Rivas.
\newblock Groups, {O}rders, and {D}ynamics, 2014.
\newblock {Available at \url{https://arxiv.org/abs/1408.5805}}.

\bibitem{Dong2022}
Ruiwen Dong.
\newblock On the identity problem for unitriangular matrices of dimension four.
\newblock In {\em 47th {I}nternational {S}ymposium on {M}athematical
  {F}oundations of {C}omputer {S}cience}, volume 241 of {\em LIPIcs. Leibniz
  Int. Proc. Inform.}, pages Art. No. 43, 14. Schloss Dagstuhl. Leibniz-Zent.
  Inform., Wadern, 2022.

\bibitem{DongThesis}
Ruiwen Dong.
\newblock {\em Algorithmic problems for subsemigroups of infinite groups}.
\newblock PhD thesis, University of Oxford, 2023.

\bibitem{Dong2023}
Ruiwen Dong.
\newblock The identity problem in {$\mathbb Z\wr\mathbb Z$} is decidable.
\newblock In {\em 50th {I}nternational {C}olloquium on {A}utomata, {L}anguages,
  and {P}rogramming}, volume 261 of {\em LIPIcs. Leibniz Int. Proc. Inform.},
  pages Art. No. 124, 20. Schloss Dagstuhl. Leibniz-Zent. Inform., Wadern,
  2023.

\bibitem{DongLICS2023}
Ruiwen Dong.
\newblock The identity problem in the special affine group of {$\mathbb Z^2$}.
\newblock In {\em Proceedings of the 38th {A}nnual {ACM}/{IEEE} {S}ymposium on
  {L}ogic in {C}omputer {S}cience ({LICS})}, page~13. IEEE Comput. Soc. Press,
  Los Alamitos, CA, 2023.

\bibitem{dong2023recent}
Ruiwen Dong.
\newblock Recent advances in algorithmic problems for semigroups.
\newblock {\em ACM SIGLOG News}, 10(4):3--23, 2023.

\bibitem{DongSODA2024}
Ruiwen Dong.
\newblock The identity problem in nilpotent groups of bounded class.
\newblock In {\em Proceedings of the 2024 {A}nnual {ACM}-{SIAM} {S}ymposium on
  {D}iscrete {A}lgorithms ({SODA})}, pages 3919--3959. SIAM, Philadelphia, PA,
  2024.

\bibitem{DongSTOC2024}
Ruiwen Dong.
\newblock Semigroup algorithmic problems in metabelian groups.
\newblock In {\em S{TOC}'24---{P}roceedings of the 56th {A}nnual {ACM}
  {S}ymposium on {T}heory of {C}omputing}, pages 884--891. ACM, New York, 2024.

\bibitem{garrabrant2017words}
Scott Garrabrant and Igor Pak.
\newblock Words in linear groups, random walks, automata and {P}-recursiveness.
\newblock {\em J. Comb. Algebra}, 1(2):127--144, 2017.

\bibitem{gray2024membership}
Robert~D. Gray and Carl-Fredrik Nyberg-Brodda.
\newblock Membership problems in braid groups and {A}rtin groups, 2024.
\newblock Preprint at \url{https://arxiv.org/abs/2409.11335}.

\bibitem{GulWeiss17}
Funda Gul and Armin Weiß.
\newblock On the dimension of matrix embeddings of torsion-free nilpotent
  groups.
\newblock {\em J. Algebra}, 477:516--539, 2017.

\bibitem{Interger_Multiplication}
David Harvey and Joris van~der Hoeven.
\newblock {Integer multiplication in time $O(n\mathrm{log}\, n)$}.
\newblock {\em Ann. of Math.}, 193(2):563 -- 617, 2021.

\bibitem{HavasWagner98}
George Havas and Clemens Wagner.
\newblock Matrix reduction algorithms for euclidean rings.
\newblock In {\em Proceedings of the 3rd ASCM}, pages 1--6, 1998.

\bibitem{Hol63}
W.~Charles Holland.
\newblock The lattice-ordered group of automorphisms of an ordered set.
\newblock {\em Michigan Math. J.}, 10:399--408, 1963.

\bibitem{HM79}
W.~Charles Holland and S.H. McCleary.
\newblock Solvability of the word problem in free lattice-ordered groups.
\newblock {\em Houston J. Math.}, 5(1):99--105, 1979.

\bibitem{Linear_programming}
Narendra Karmarkar.
\newblock A new polynomial-time algorithm for linear programming.
\newblock {\em Combinatorica}, 4, 1984.

\bibitem{Kop82}
Valerii~M. Kopytov.
\newblock Nilpotent lattice-ordered groups.
\newblock {\em Siberian Math. J.}, 23:690--693, 1982.

\bibitem{lohrey2024membership}
Markus Lohrey.
\newblock Membership problems in infinite groups.
\newblock In {\em Conference on Computability in Europe}, pages 44--59.
  Springer, 2024.

\bibitem{subgroup_pres}
Jeremy Macdonald, Alexei Myasnikov, Andrey Nikolaev, and Svetla Vassileva.
\newblock Logspace and compressed-word computations in nilpotent groups.
\newblock {\em Transactions of the American Math.\ Soc.}, 375:5425--5459, 2022.

\bibitem{Mal51}
Anatoli~I. Malcev.
\newblock On the full ordering of groups.
\newblock {\em Trudy Mat. Inst. Steklov}, 38(173--175), 1951.

\bibitem{McC82}
Stephen~H. McCleary.
\newblock The word problem in free normal valued lattice-ordered groups: a
  solution and practical shortcuts.
\newblock {\em Algebra Universalis}, 14:317--348, 1982.

\bibitem{MPT23}
George Metcalfe, Francesco Paoli, and Constantine Tsinakis.
\newblock {\em Residuated Structures in Algebra and Logic}, volume 277 of {\em
  Mathematical Surveys and Monographs}.
\newblock American Mathematical Society, 2023.

\bibitem{pak_soukup}
Igor Pak and David Soukup.
\newblock Algebraic and arithmetic properties of the cogrowth sequence of
  nilpotent groups, 2022.
\newblock Preprint at \url{https://arxiv.org/abs/2210.09419}.

\bibitem{Rhe72}
Akbar~H. Rhemtulla.
\newblock Right-ordered groups.
\newblock {\em Canadian J. Math}, 24:891--895, 1972.

\bibitem{Rivas_BS12}
Cristóbal Rivas.
\newblock On spaces of {C}onradian group orderings.
\newblock {\em J. Group Theory}, 13(3):337--353, 2010.

\bibitem{Orders_virt_solvable}
Cristóbal Rivas and Romain Tessera.
\newblock On the space of left-orderings of virtually solvable groups.
\newblock {\em Groups Geom. Dyn.}, 10(1):65--90, 2016.

\bibitem{romankov2022undecidability}
Vitaly Roman'kov.
\newblock Undecidability of the submonoid membership problem for a sufficiently
  large finite direct power of the {H}eisenberg group, 2022.
\newblock Preprint at \url{https://arxiv.org/abs/2209.14786}.

\bibitem{shafrir2024saturation}
Doron Shafrir.
\newblock A saturation theorem for submonoids of nilpotent groups and the
  identity problem, 2024.
\newblock Preprint at \url{https://arxiv.org/abs/2402.07337}.

\bibitem{Stewart70}
Ian~N. Stewart.
\newblock An algebraic treatment of {M}alcev's theorems concerning nilpotent
  {L}ie groups and their {L}ie algebras.
\newblock {\em Compositio Math.}, 22:289--312, 1970.

\bibitem{Ziegler}
Günter~M. Ziegler.
\newblock {\em Lectures on Polytopes}.
\newblock Graduate Texts in Mathematics. Springer New York, 1994.

\end{thebibliography}

\end{document}